\documentclass[reqno,a4paper]{amsart}
\usepackage{amssymb}
\usepackage{amsmath}
\usepackage{amsfonts}

\newtheorem{theorem}{Theorem}[section]
\theoremstyle{plain}

\newtheorem{corollary}{Corollary}[section]

\newtheorem{definition}{Definition}[section]

\newtheorem{lemma}{Lemma}[section]

\numberwithin{equation}{section}
\input{tcilatex}

\begin{document}
\title[]{ON WEAK $r$-HELIX SUBMANIFOLDS }
\author{Evren Z\i plar}
\address{Department of Mathematics, Faculty of Science, University of
Ankara, Tando\u{g}an, Turkey}
\email{evrenziplar@yahoo.com}
\urladdr{}
\author{Ali \c{S}enol}
\address{Department of Mathematics, Faculty of Science, \c{C}ank\i r\i\
Karatekin University, \c{C}ank\i r\i , Turkey}
\email{asenol@karatekin.edu.tr}
\author{Yusuf Yayl\i }
\address{Department of Mathematics, Faculty of Science, University of
Ankara, Tando\u{g}an, Turkey}
\email{yayli@science.ankara.edu.tr}
\thanks{}
\urladdr{}
\date{}
\keywords{Weak $r$-helix submanifold; Line of curvature; Asymptotic curve;
Helix line.\\
Corresponding author: Evren Z\i plar, e-mail: evrenziplar@yahoo.com}
\thanks{}

\begin{abstract}
In this paper, we investigate special curves on a weak $r$-helix submanifold
in Euclidean $n$-space $E^{n}$. Also, we give the important relations
between weak $r$-helix submanifolds and the special curves such as line of
curvature, asymptotic curve and helix line.
\end{abstract}

\maketitle

\section{Introduction}

In differential geometry of manifolds, an helix submanifold of $%
\mathbb{R}
^{n}$ with respect to a fixed direction $d$ in $%
\mathbb{R}
^{n}$ is defined by the property that tangent planes make a constant angle
with the fixed direction $d$ (helix direction) in [4]. Di Scala and Ruiz-Hern%
\'{a}ndez have introduced the concept of these manifolds in [4]. Besides,
the concept of weak $r$-helix submanifold of $%
\mathbb{R}
^{n}$ was introduced in [3]. Let $M\subset 
\mathbb{R}
^{n}$ be a submanifold. We say that $M$ is a weak $r$-helix if there exist $%
r $ linearly independent directions $d_{1},...,d_{r}$, such that $M$ \ is a
helix with respect to every $d_{j}$[5].

Recently, M. Ghomi worked out the shadow problem given by H.Wente. And, He
mentioned the shadow boundary in [8]. Ruiz-Hern\'{a}ndez investigated that
shadow boundaries are related to helix submanifolds in [12].

Helix hypersurfaces have been worked in nonflat ambient spaces in [6,7].
Cermelli and Di Scala have also studied helix hypersurfaces in liquid
cristals in [2].

This paper is organized as follows. In section 2, we will give some basic
properties in the general theory of weak $r$-helix, helix submanifolds and
curves. And, in section 3, we will give the important relations between weak 
$r$-helix submanifolds and some special curves such as line of curvature,
asymptotic curve and helix lines.

\section{Basic Properties}

\begin{definition}
Given a submanifold $M\subset 
\mathbb{R}
^{n}$ and an unitary vector $d$ in $%
\mathbb{R}
^{n}$, we say that $M$ is a helix with respect to $d$ if for each $q\in M$
the angle between $d$ and $T_{q}M$ is constant.

Let us recall that a unitary vector $d$ can be decomposed in its tangent and
orthogonal components along the submanifold $M$, i.e. $d=\cos (\theta
)T+\sin (\theta )\xi $ with $\left\Vert T\right\Vert =\left\Vert \xi
\right\Vert =1$, where $T\in TM$ and $\xi \in \vartheta (M)$.The angle
between $d$ and $T_{q}M$ is constant if and only if the tangential component
of $d$ has constant length $\left\Vert \cos (\theta )T\right\Vert =\cos
(\theta )$. We can assume that $0<\theta <\frac{\pi }{2}$ and we can say
that $M$ is a helix of angle $\theta $.

We will call $T$ and $\xi $ the tangent and normal directions of the helix
submanifold $M$. We can call $d$ the helix direction of $M$ and we will
assume $d$ always to be unitary [5].
\end{definition}

\begin{definition}
Let $M\subset 
\mathbb{R}
^{n}$ be a helix submanifold of angle $\theta \neq \frac{\pi }{2}$ w.r. to
the direction $d\in 
\mathbb{R}
^{n}$. We will call the integral curves of the tangent direction $T$ of the
helix $M$, the helix lines of $M$ w.r.to $d$. Moreover, we say that a helix
submanifold $M\subset 
\mathbb{R}
^{n}$ is a ruled helix if all the helix lines of $M$ are straight lines [5].
\end{definition}

\noindent \textbf{Proposition 2.1 }\textit{The helix lines of a helix
submanifold }$M\subset 
\mathbb{R}
^{n}$\textit{\ are geodesic in }$M$\textit{\ [5].}

\begin{definition}
A submanifold $M\subset 
\mathbb{R}
^{n}$ is a weak $r$-helix if there exist $r$ linearly independent directions 
$d_{1},...,d_{r},$ such that $M$ is a helix with respect to every $d_{j}$
[5].
\end{definition}

\noindent \textbf{Remark 2.1 }\textit{We say that }$\xi $\textit{\ is
parallel normal in the direction }$X\in TM$\textit{\ if }$\ \nabla
_{X}^{\perp }\xi =0$\textit{. Here, }$\nabla ^{\perp }$\textit{\ denotes the
normal connection of }$M$\textit{\ induced by the standard covariant
derivative of the Euclidean ambient. Let us denote by }$D$\textit{\ the
standard covariant derivative in }$%
\mathbb{R}
^{n}$\textit{\ and by }$\nabla $\textit{\ the induced covariant derivative
in }$M$\textit{. Let }$A^{\xi }$\textit{\ and }$V$\textit{\ be the shape
operator and the second fundamental form of }$M\subset 
\mathbb{R}
^{n}$\textit{\ [5]. }

\begin{definition}
Let $M$ be a submanifold of the Riemannian manifold $%
\mathbb{R}
^{n}$ and let $D$ be the Riemannian connexion on $%
\mathbb{R}
^{n}$. For $C^{\infty \text{ }}$fields $X$ and $Y$ with domain $A$ on $M$
(and tangent to $M$), define $\nabla _{X}Y$ and $V(X,Y)$ on $A$ by
decomposing $D_{X}Y$ into unique tangential and normal components,
respectively; thus,%
\begin{equation*}
D_{X}Y=\nabla _{X}Y+V(X,Y)\text{. }
\end{equation*}%
Then, $\nabla $ is the Riemannian connexion on $M$ and $V$ is a symmetric
vector-valued 2-covariant $C^{\infty \text{ }}$tensor called the second
fundamental tensor. The above composition equation is called the Gauss
equation [9].
\end{definition}

\begin{definition}
Let $M$ be a submanifold of the Riemannian manifold $%
\mathbb{R}
^{n}$ , let $D$ be the Riemannian connexion on $%
\mathbb{R}
^{n}$ and let $\nabla $ be the Riemannian connexion on $M$. Then, the
formula of Weingarten%
\begin{equation*}
D_{X}\xi =-A^{\xi }(X)+\nabla _{X}^{\bot }\xi
\end{equation*}%
for every $X$ tangent to $M$ and for every $\xi $ normal to $M$. $A^{\xi }$
is the shape operator associated to $\xi $ also known as the Weingarten
operator corresponding to $\xi $ and $\nabla ^{\bot }$ is the induced
connexion in the normal bundle of $M$ . $A^{\xi }(X)$ is also the tangent
component of $-D_{X}\xi $ and will be denoted by $A^{\xi }(X)=$tang$%
(-D_{X}\xi $ $)$[10,11].
\end{definition}

\noindent \textbf{Remark 2.2 }\textit{Let us observe that for any helix
euclidean submanifold }$M$\textit{, the following system holds for every }$%
X\in TM$\textit{, where the helix direction }$d=\cos (\theta )T+\sin (\theta
)\xi $\textit{.\medskip \vspace{0in} }%
\begin{equation}
\cos (\theta )\nabla _{X}T-\sin (\theta )A^{\xi }(X)=0
\end{equation}%
\begin{equation}
\cos (\theta )V(X,T)+\sin (\theta )\nabla _{X}^{\bot }\xi =0
\end{equation}%
[5].

\begin{definition}
If $\alpha $ is a (unit speed) curve in $M$ with $C^{\infty \text{ }}$unit
tangent $T$, then $V(T,T)$ is called normal curvature vector field of $%
\alpha $ and $k_{T}=\left\Vert V(T,T)\right\Vert $ is called the normal
curvature of $\alpha $ [9].
\end{definition}

\section{Main Theorems and Definitions}

\begin{theorem}
Let $M\subset 
\mathbb{R}
^{n}$ be a weak $r$-helix submanifold with respect to the directions $%
d_{j}\in 
\mathbb{R}
^{n}$, $j=1,...,r$. Let $D$ be Riemannian connexion (standard covariant
derivative) on $E^{n}$ and $\nabla $ be Riemannian connexion on $M$. Let us
assume that $\alpha :I\subset IR\rightarrow M$ is a unit speed (parametrized
by arc length function $s$) curve on $M$ with unit tangent $T$ . Then,the
normal component $\xi _{j}$ of $d_{j}$ is parallel normal in the direction $T
$ if and only if  $T_{j}^{^{\prime }}\in $ $TM$ along the curve $\alpha $,
where $T_{j}$ is the unit tangent component of the direction $d_{j}$.
\end{theorem}

\begin{proof}
We assume that the normal component $\xi _{j}$ of $d_{j}$ is parallel normal
in the direction $T$. Since $T$ and $T_{j}\in TM$, from the Gauss equation
in Definition (2.4),%
\begin{equation}
D_{T}T_{j}=\nabla _{T}T_{j}+V(T,T_{j})
\end{equation}%
According to the Theorem, since the normal component $\xi _{j}$ of $d_{j}$
is parallel normal in the direction $T$, i.e.$\nabla _{T}^{\bot }\xi _{j}=0$
(see Remark 2.1), from (2.2) in Remark 2.2 ($0<\theta <\frac{\pi }{2}$) 
\begin{equation}
V(T,T_{j})=0
\end{equation}%
So, by using (3.1),(3.2) and Frenet formulas, we have:%
\begin{equation*}
D_{T}T_{j}=\frac{dT_{j}}{ds}=T_{j}^{^{\prime }}=\nabla _{T}T_{j}\text{.}
\end{equation*}%
That is, the vector field $T_{j}^{^{\prime }}\in T_{\alpha (t)}M$, where $%
T_{\alpha (t)}M$ is the tangent space of $M$.

Conversely, let us assume that $T_{j}^{^{\prime }}\in $ $TM$ along the curve 
$\alpha $. Then, from Gauss equation, $V(T,T_{j})=0$. Hence, from (2.2) in
Remark 2.2 ($0<\theta <\frac{\pi }{2}$), $\nabla _{T}^{\bot }\xi _{j}=0$ .
That is, the normal component $\xi _{j}$ of $d_{j}$ is parallel normal in
the direction $T$. This completes the proof.
\end{proof}

\begin{theorem}
Let $M\subset 
\mathbb{R}
^{n}$ be a weak $r$-helix submanifold with respect to the directions $%
d_{j}\in 
\mathbb{R}
^{n}$, $j=1,...,r$. Let $D$ be Riemannian connexion (standard covariant
derivative) on $E^{n}$ and $\nabla $ be Riemannian connexion on $M$. Then,
the normal curvature of the unit integral curve $\alpha _{j}$ of $T_{j}$
equals $\left\vert k_{j}\right\vert $, where $T_{j}$ is the unit tangent
component of the direction $d_{j}$ and $k_{j}$ is the first curvature of $%
\alpha _{j}$.
\end{theorem}

\begin{proof}
Since $\alpha _{j}$ is the integral curve of $T_{j}\in TM$, we can write%
\begin{equation*}
\frac{d\alpha _{j}}{ds}=T_{j}
\end{equation*}%
along the curve $\alpha _{j}$. Also, since $T_{j}\in TM$, from the Gauss
equation in Definition (2.4)%
\begin{equation}
D_{T_{j}}T_{j}=\nabla _{T_{j}}T_{j}+V(T_{j},T_{j})\text{.}
\end{equation}%
On the other hand, the integral curves of $T_{j}$ are geodesics in $M$
according to the Proposition 2.1. That is%
\begin{equation}
\nabla _{T_{j}}T_{j}=0\text{.}
\end{equation}%
So, by using (3.3),(3.4) and Frenet formulas, we have:%
\begin{eqnarray*}
D_{T_{j}}T_{j} &=&\frac{dT_{j}}{ds}=T_{j}^{^{\prime }} \\
&=&k_{j}V_{2_{j}}\text{ (}V_{2_{j}}\text{ is the unit principal normal of }%
\alpha _{j}\text{)} \\
&=&V(T_{j},T_{j})\text{.}
\end{eqnarray*}%
Hence, we get $\left\Vert V(T_{j},T_{j})\right\Vert =\left\Vert
k_{j}V_{2_{j}}\right\Vert =$ $\left\vert k_{j}\right\vert $.This completes
the proof.
\end{proof}

\begin{lemma}
Let $M\subset 
\mathbb{R}
^{n}$ be a weak $r$-helix submanifold with respect to the directions $%
d_{j}\in 
\mathbb{R}
^{n}$, $j=1,...,r$. Let $D$ be the Riemannian connexion on $%
\mathbb{R}
^{n}$ and let $\nabla $ be the Riemannian connexion on $M$. If $\alpha
_{j}:I\subset IR\rightarrow M$ is the (unit speed) integral curve of the
unit tangent direction $T_{j}$ of $d_{j}$, then $V_{2_{j}}\in \vartheta (M)$%
, where $V_{2_{j}}$ is the unit principal normal of $\alpha _{j}$ and $%
\vartheta (M)$ is the normal space of $M$.
\end{lemma}

\begin{proof}
Since $T_{j}\in TM$, from the Gauss equation in Definition 2.4,%
\begin{equation}
D_{T_{j}}T_{j}=\nabla _{Tj}T_{j}+V(T_{j},T_{j})
\end{equation}%
According to the Proposition 2.1, since $\alpha _{j}$ is a geodesic curve on 
$M$,%
\begin{equation}
\nabla _{T_{j}}T_{j}=0
\end{equation}%
So, by using (3.5), (3.6) and Frenet formulas, we get:%
\begin{equation*}
D_{T_{j}}T_{j}=k_{j}V_{2_{j}}=V(T_{j},T_{j})\text{. (}k_{j}\text{ is the
first curvature of }\alpha _{j}\text{)}
\end{equation*}%
That is, the vector field $V_{2_{j}}\in \vartheta (M)$ along the curve $%
\alpha _{j}$. This completes the proof.
\end{proof}

\begin{definition}
Let $M$ be a submanifold of the Riemannian manifold of $%
\mathbb{R}
^{n}$and let $A^{\xi }$ be the shape operator in a direction $\xi \in $\ \ $%
\vartheta (M)$. For a vector field $X\in TM$, if $\left\langle A^{\xi
}(X),X\right\rangle =0$, then $X$ will be called asymptotic in the direction
\ $\xi $.
\end{definition}

\begin{definition}
Let $M$ be a submanifold of the Riemannian manifold of $%
\mathbb{R}
^{n}$ and let $\alpha :I\subset IR\rightarrow M$ be a unit speed curve in $M$%
. If $\left\langle A^{\xi }(T),T\right\rangle =0$, then the curve $\alpha $
will be called an asymptotic curve in the direction $\xi $, where $T$ unit
tangent vector field of $\alpha $ and $\xi \in $\ \ $\vartheta (M)$ (normal
space).
\end{definition}

\begin{definition}
The second normal space of $M$ $\subset 
\mathbb{R}
^{n}$ consist of the normal vectors, $\xi \in $\ \ $\vartheta (M)$, such
that the shape operator in its direction is zero, i.e. $A^{\xi }=0$ [5].
\end{definition}

\begin{theorem}
Let $M$ be a submanifold of the Riemannian manifold of $%
\mathbb{R}
^{n}$. Then every curve in $M$ is asymptotic in the direction $\xi $ if $\xi
\in $\ \ $\vartheta (M)$ is an element of the second normal space $M$ $%
\subset 
\mathbb{R}
^{n}$.
\end{theorem}

\begin{proof}
Let assume that $\xi \in $\ \ $\vartheta (M)$ is an element of the second
normal space $M$ $\subset 
\mathbb{R}
^{n}$ and let $\alpha $ be an arbitrary curve with the unit tangent $T$ in $%
M $. Then, from Definition 3.3, $A^{\xi }=0$ for every $X\in TM$. And, in
particular since $T\in TM$, $A^{\xi }\left( T\right) =0$. Hence, we deduce
that $\left\langle A^{\xi }(T),T\right\rangle =0$. Consequently, since $%
\alpha $ is an arbitrary curve, from Definition 3.2, every curve in $M$ is
asymptotic in the direction $\xi $. This completes the proof.
\end{proof}

\begin{theorem}
Let $M\subset 
\mathbb{R}
^{n}$ be a weak $r$-helix submanifold with respect to the directions $%
d_{j}\in 
\mathbb{R}
^{n}$, $j=1,...,r$. Let $D$ be the Riemannian connexion on $%
\mathbb{R}
^{n}$ and let $\nabla $ be the Riemannian connexion on $M$. If $T_{j}$ is
parallel in $M$ ,i.e. $\nabla _{X}T_{j}=0$ for every $X\in TM$, then every $%
X\in TM$ is asymptotic in the direction $\xi _{j}$. Here, $T_{j}$ is tangent
component of $d_{j}$ and $\xi _{j}$ is normal component of $d_{j}$.
Conversely, if every $X\in TM$ is asymptotic in the direction $\xi _{j}$,
then $T_{j}$ is parallel in $M$.
\end{theorem}

\begin{proof}
We assume that $T_{j}$ is parallel in $M$. That is, $\nabla _{X}T_{j}=0$ for
every $X\in TM$. So, by using the equation (2.1), we deduce that $A^{\xi
_{j}}(X)=0$ for every $X\in TM$ ($\theta \neq 0$). Due to the fact that $%
A^{\xi _{j}}(X)=0$ for every $X\in TM$, $\left\langle A^{\xi
_{j}}(X),X\right\rangle =0$ for every $X\in TM$. Consequently, by using the
Definition 3.1, every $X\in TM$ is asymptotic in the direction $\xi _{j}$.

Conversely, let us assume that every $X\in TM$ is asymptotic in the
direction $\xi _{j}$. Then, $\left\langle A^{\xi _{j}}(X),X\right\rangle =0$
for every $X\in TM$. Moreover, by using the equation (2.1), we obtain $%
\left\langle \nabla _{X}T_{j},X\right\rangle =0$ for every $X\in TM$ ($%
\theta \neq \frac{\pi }{2}$). Hence, $\ $we have $\nabla _{X}T_{j}=0$. This
completes the proof.
\end{proof}

\begin{theorem}
Let $M\subset 
\mathbb{R}
^{n}$ be a weak $r$-helix submanifold with respect to the directions $%
d_{j}\in 
\mathbb{R}
^{n}$, $j=1,...,r$ and let $D$ be the Riemannian connexion on $%
\mathbb{R}
^{n}$. If $\alpha _{j}:I\subset IR\rightarrow M$ is the (unit speed)integral
curve of the unit tangent direction $T_{j}$ of $d_{j}$, then the curve $%
\alpha _{j}$ is asymptotic in the direction $\xi _{j}$, where $\xi _{j}$ is
orthogonal component of $d_{j}$.
\end{theorem}

\begin{proof}
Since $M$ is a weak $r$-helix submanifold, we can decompose the direction $%
d_{j}$ in its tangent and normal components:%
\begin{equation}
d_{j}=\cos (\theta _{j})T_{j}+\sin (\theta _{j})\xi _{j}
\end{equation}%
From (3.7), by taking derivatives on both sides along the curve $\alpha _{j}$%
, we have:%
\begin{equation*}
0=\cos (\theta _{j})T_{j}^{^{\prime }}+\sin (\theta _{j})\xi _{j}^{^{\prime
}}
\end{equation*}%
and by using Frenet formulas, we get:%
\begin{equation}
0=\left( k_{j}\cos (\theta _{j})\right) V_{2_{j}}+\sin (\theta _{j})\xi
_{j}^{^{\prime }}
\end{equation}%
According to the Lemma 3.1, since $V_{2_{j}}\in \vartheta (M)$ along the
curve $\alpha _{j}$, from (3.8), we deduce that $\xi _{j}^{^{\prime }}\in
\vartheta (M)$. On the other hand,%
\begin{eqnarray*}
A^{\xi _{j}}(T_{j}) &=&\text{tang}(-D_{T_{j}}\xi _{j}) \\
&=&\text{tang}(-\xi _{j}^{^{\prime }})
\end{eqnarray*}%
and since $\xi _{j}^{^{\prime }}\in \vartheta (M)$, we obtain tang$(-\xi
_{j}^{^{\prime }})=0$. So $\left\langle A^{\xi
_{j}}(T_{j}),T_{j}\right\rangle =0$. This completes the proof.
\end{proof}

\begin{definition}
Given an Euclidean submanifold of arbitrary codimension $M\subset 
\mathbb{R}
^{n}$. A curve $\alpha $ in $M$ is called a line of curvature if its tangent 
$T$ is a principal vector at each of its points. In other words, when $T$
(the tangent of $\alpha $) is a principal vector at each of its points, for
an arbitrary normal vector field $\xi \in \vartheta (M)$, the shape operator 
$A^{\xi }$ associated to $\xi $ says $A^{\xi }(T)=$tang$(-$ $D_{T}\xi
)=\lambda _{j}T$ along the curve $\alpha $, where $\lambda _{j}$ is a
principal curvature and $D$ be the Riemannian connexion(standard covariant
derivative) on $%
\mathbb{R}
^{n}$ [1].
\end{definition}

\begin{theorem}
Let $M\subset 
\mathbb{R}
^{n}$ be a weak $r$-helix submanifold with respect to the directions $%
d_{j}\in 
\mathbb{R}
^{n}$, $j=1,...,r$ and let $D$ be the Riemannian connexion on $%
\mathbb{R}
^{n}$. Let us assume that $\alpha :I\subset 
\mathbb{R}
\rightarrow M$ is a (unit speed) line of curvature (not a straight line) for
a normal vector field $\xi \in \vartheta (M)$, where $\xi ^{%
{\acute{}}%
}$ is $TM$ along the curve $\alpha $. Then, $d_{j}\notin Sp\left\{ \xi
,T\right\} $ along the curve $\alpha $ for all the directions $d_{j}$, where 
$T$ is the unit tangent vector field of $\alpha $.
\end{theorem}

\begin{proof}
We assume that $d_{j}\in Sp\left\{ \xi ,T\right\} $ along the curve $\alpha $
for any direction $d_{j}$.Since $M$ is a weak $r$-helix submanifold, we can
decompose the direction $d_{j}$ in its tangent and normal components:%
\begin{equation}
d_{j}=\cos (\theta _{j})\xi +\sin (\theta _{j})T\text{,}
\end{equation}%
where $\theta _{j}$ is constant. From (3.9), by taking derivatives on both
sides along the curve $\alpha $, we get:%
\begin{equation}
0=\cos (\theta _{j})\xi ^{%
{\acute{}}%
}+\sin (\theta _{j})T^{%
{\acute{}}%
}
\end{equation}%
Moreover, since $\alpha $ is a line of curvature (not a straight line) for a
normal vector field $\xi \in \vartheta (M)$,%
\begin{equation*}
A^{\xi }(T)=tang(-D_{T}\xi )=\text{tang}(-\xi ^{%
{\acute{}}%
})=\lambda _{j}T
\end{equation*}%
along the curve $\alpha $. According to the Theorem, since $\xi ^{%
{\acute{}}%
}$ is $TM$ along the curve $\alpha $, 
\begin{equation}
\text{tang}(-\xi ^{%
{\acute{}}%
})=-\xi ^{%
{\acute{}}%
}=\lambda _{j}T
\end{equation}%
By using the equations (3.10) and (3.11), we deduce that the system $\left\{
T,T^{%
{\acute{}}%
}\right\} $ is linear dependent. But, the system $\left\{ T,T^{%
{\acute{}}%
}\right\} $ is never linear dependent.This is a contradiction. This
completes the proof.
\end{proof}

This latter Theorem has the following corollary.

\begin{corollary}
For an arbitrary direction $d_{j}\in 
\mathbb{R}
^{n}$, if $d_{j}\in Sp\left\{ \xi ,T\right\} $, then the curve $\alpha $ is
not a line of curvature with respect to $\xi \in \vartheta (M)$, where $T$
is the unit tangent vector field of $\alpha $.
\end{corollary}

\begin{theorem}
Let $M\subset 
\mathbb{R}
^{n}$ be a full submanifold (it is not contained in a hyperplane of the
ambient $%
\mathbb{R}
^{n}$) which is a helix with respect to the direction $d$. Let $\xi $ be the
normal component of $d$, i.e. $d=\cos (\theta )T+\sin (\theta )\xi $. Then $%
M $ is a ruled helix(see definition 2.2) if and only if $\xi $ is $\nabla
_{T}^{\bot }\xi =0$ [5].
\end{theorem}

\begin{theorem}
Let $M\subset 
\mathbb{R}
^{n}$ be a full submanifold which is a helix with respect to the direction $%
d=\cos (\theta )T+\sin (\theta )\xi $ ($\theta \neq 0$). Then, $M$ is a
ruled helix if and only if the normal curvatures of the helix lines of $M$
with respect to $d$ equal zero.
\end{theorem}

\begin{proof}
We assume that $M$ is a ruled helix. Then, from Theorem 3.7, $\nabla
_{T}^{\bot }\xi =0$ for the direction $d$. So, from (2.2) in Remark 2.2, $%
V(T,T)=0$ (since $M$ is full or $\theta \neq \frac{\pi }{2}$). Hence, the
normal curvatures of the helix lines of $M$ with respect to $d$ equal zero.

Conversely, let us assume that the normal curvatures of the helix lines of $%
M $ with respect to $d$ equal zero. In other words, $V(T,T)=0$. Since $%
\theta \neq 0$, from (2.2) in Remark 2.2, we obtain $\nabla _{T}^{\bot }\xi
=0$. And, from Theorem 3.7, $M$ is a ruled helix. This completes the proof.
\end{proof}

This latter Theorem has the following corollary.

\begin{corollary}
In Theorem 3.8, in particular, let us assume that $M$ is a hypersurface in $%
\mathbb{R}
^{n}$. Then, since the helix lines of $M$ are straight lines of $%
\mathbb{R}
^{n}$ (see Lemma 2.5 in [4]), $M$ is always a ruled helix. Finally, the
normal curvatures of the helix lines of $M$ with respect to $d$ equal always
zero.
\end{corollary}

\noindent \textbf{Acknowledgment.} The authors would like to thank referees
for their valuable suggestions and comments that helped to improve the
presentation of this paper.

\end{document}